\documentclass[final, a4paper, 12pt]{article}

\usepackage{amsfonts}
\usepackage{amsmath}
\usepackage{amssymb}
\usepackage{amscd}
\usepackage{amsthm}

\usepackage[mathscr]{euscript}
\usepackage{color}
\usepackage{fancyhdr}
\usepackage{dsfont}
\usepackage[a4paper,scale=0.752]{geometry}
\usepackage{hyperref} 
\usepackage{bm}
\usepackage{cite}
\usepackage{showkeys}

\usepackage{dblaccnt}
\usepackage{accents}
\usepackage{graphicx}
\usepackage{array}

\usepackage{hyperref}

\newtheorem{definition}{Definition}

\newtheorem{theorem}[definition]{Theorem}
\newtheorem{corollary}[definition]{Corollary}

\newtheorem{assumption}[definition]{Assumption}

\newtheorem{remark}[definition]{Remark}

\newcommand*{\N}{\ensuremath{\mathbb{N}}}
\newcommand*{\Z}{\ensuremath{\mathbb{Z}}}
\newcommand*{\R}{\ensuremath{\mathbb{R}}}
\newcommand*{\C}{\ensuremath{\mathbb{C}}}
\renewcommand{\i}{\mathrm{i}}
\renewcommand{\phi}{\varphi}
\renewcommand{\rho}{{\varrho}}
\renewcommand{\epsilon}{{\varepsilon}}

\renewcommand{\d}[1]{\,\mathrm{d}#1 \,}

\newcommand{\I}{\mathcal{I}}
\newcommand{\A}{\mathcal{A}}

\newcommand{\J}{\mathcal{J}} 
\newcommand{\0}{{0}} 

\renewcommand{\S}{\mathcal{S}}

\newcommand{\grad}{\nabla}

\newcommand{\W}{{W_{\hspace*{-1pt}\Lambda}}} 
\newcommand{\Wast}{{W_{\hspace*{-1pt}\Lambda^\ast}}} 

\newlength{\dhatheight}

\usepackage{color}
\newcommand{\high}[1]{} 

\begin{document}

\sloppy

\title{The Study of the Bloch Transformed Fields Scattered by Locally Perturbed Periodic Surfaces}
\author{Ruming Zhang\thanks{Center for Industrial Mathematics, University of Bremen
; \texttt{rzhang@uni-bremen.de}}}
\date{}

\maketitle

\begin{abstract}
Scattering problems from locally perturbed periodic surfaces have been studied both theoretically and numerically in recent years. In this paper, we will discuss more details of the structure of the Bloch transformed total fields. The idea is inspired by Theorem a in \cite{Kirsc1993}, which considered how the total field depends on the incident plane waves, from a smooth enough periodic surface. We will show that when the incident field satisfies some certain conditions, the Bloch transform of the total field depends analytically on the quasi-periodicities in one periodic cell except for a countable number of points, while in the neighborhood of each point, a square-root like singularity exists. We also give some examples to show that the conditions are satisfied by a large number of commonly used incident fields. This result also provides a probability to improve the numerical solution of this kind of problems, which is expected to be discussed in our future papers.
\end{abstract}

\section{Introduction}

The problems that quasi-periodic incident fields scattered by periodic surfaces have been well studied in the past few years. A classical way to handle this kind of problems, which are defined in unbounded domains, is to reduce the problems into one periodic cell,  see \cite{Kirsc1993,Bao1994, Bao1995a, Bao2000}. However, when the incident field is no longer (quasi-) periodic, or the surface is no longer periodic, the scattered field is no longer (quasi-)periodic,  the classical way fails to work, so new  methods are required for these more difficult problems. Another way is to treat this kind of problems as incident fields scattered by rough surfaces, see \cite{Chand1998,Chand1998a,Chand1999,Chand1999a}. In recent years, the Floquet-Bloch based method was applied to analyze the well-posedness of this kind of problems,  see \cite{Coatl2012,Hadda2016,Lechl2016}.  The Bloch transform build up a "bridge" between the non-periodic scattering problem and a family of quasi-periodic scattering problems, which also provided a computational method for numerical solutions. In these papers, the properties of the Bloch transformed total fields are important both theoretically and numerically, thus we will discuss some more properties in this paper.

This Floquet-Bloch based method was studied theoretically in \cite{Lechl2015e} and \cite{Lechl2016}, for non-periodic incident fields, e.g. Herglotz wave functions or Green's functions, scattered by (locally perturbed) periodic surfaces. Based on the theoretical results, a numerical method was established to solve the scattering problems with a non-periodic incident field with a periodic surface, for cases in 2D space, see \cite{Lechl2016a} and for  3D space, see \cite{Lechl2016b}. Later on, the numerical method was extended for scattering problems with a locally perturbed surface in 2D space in \cite{Lechl2017}. The method was also applied to locally perturbed periodic waveguide problems in \cite{Hadda2016}. 

In the numerical solution of the scattering problems with (locally perturbed) periodic surfaces, for a chosen discrete space, the convergence rate depends on the regularity of the Bloch transformed total field. To improve the numerical method, a more detailed study of the regularity is required. In the paper \cite{Kirsc1993}, Kirsch has studied how the total fields depend on the incident angle and wave number of incident plane waves, from smooth enough periodic surfaces. The method adopted in this paper is the perturbation theory. Inspired by this idea, we may extend the result to the dependence of the quasi-periodicity of the Bloch transformed total field, when the incident field and its Bloch transform satisfy some certain conditions, and the surface is locally perturbed. In future, it is hopefully that the result could be helpful to improve the numerical method to solve such kind of scattering problems, or might be extended to solve problems in higher dimensional spaces.

The rest of the paper is organized as follows. In Section 2,  the known results of the locally perturbed periodic problems will be shown. In Section 3, we will study extend the result in \cite{Kirsc1993} to the non-perturbed cases, and finally, the result will be extended to the locally perturbed case in Section 4. For a more general case, we will show some results in Section 5. The definition and some properties of the Bloch transform will be shown in the Appendix.

\section{Locally perturbed periodic problems}\label{sec:review}

\begin{figure}[htb]
\centering
\includegraphics[width=0.7\textwidth]{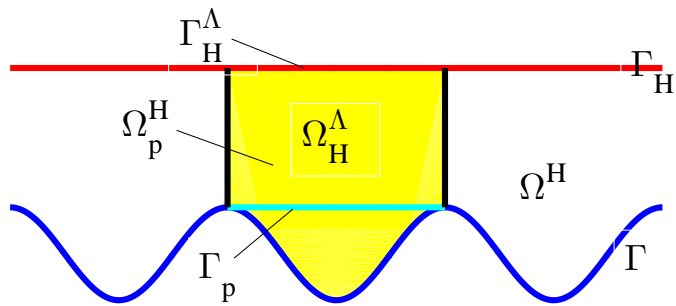}
\caption{Physical structures of the scattering problems.}
\end{figure}
\subsection{Notations and domains}

Let $\zeta$ be a Lipschitz continuous periodic function with period $\Lambda$, and the periodic surface is defined by
\begin{equation*}
\Gamma:=\left\{(x_1,\zeta(x_1)):\,x_1\in\R\right\}.
\end{equation*} 
Let $\Lambda^*=2\pi/\Lambda$, define the periodic cells by
\begin{equation*}
\W=\left(-\frac{\Lambda}{2},\frac{\Lambda}{2}\right],\quad\Wast=\left(-\frac{\Lambda^*}{2},\frac{\Lambda^*}{2}\right]=\left(-\frac{\pi}{\Lambda},\frac{\pi}{\Lambda}\right].
\end{equation*}
Let $\zeta_p$ be a Lipschitz local perturbation of $\zeta$ such that $\rm{supp}(\zeta_p-\zeta)$ is a compact domain in $\R$. For simplicity, assume that $\rm{supp}(\zeta_p-\zeta)\subset \W$, then define the locally perturbed surface by
\begin{equation*}
\Gamma_p:=\{(x_1,\zeta_p(x_1)):\,x_1\in\R\}.
\end{equation*}
We can define the space above the surfaces by
\begin{eqnarray*}
&&\Omega:=\{(x_1,x_2):\,x_1\in\R,\,x_2>\zeta(x_1)\}\\
&&\Omega_p:=\{(x_1,x_2):\,x_1\in\R,\,x_2>\zeta_p(x_1)\}.
\end{eqnarray*}
For simplicity, in this paper, we assume that $\Gamma$ and $\Gamma_p$ are all above the straight line $\{x_2=0\}$, and $H$ be a positive number such that $H>\max\{\|\zeta\|_\infty,\|\zeta_p\|_\infty\}$. Define the line $\Gamma_H=\R\times\{H\}$ and the spaces with between the surfaces $\Gamma$ or $\Gamma_p$ and $\Gamma_H$
\begin{eqnarray*}
&&\Omega_H:=\{(x_1,x_2):\,x_1\in\R,\,\zeta(x_1)<x_2<H\}\\
&&\Omega^p_H:=\{(x_1,x_2):\,x_1\in\R,\,\zeta_p(x_1)<x_2<H\}.
\end{eqnarray*}
We can also define the curves and domains in one periodic cell: 
\begin{eqnarray*}
&&\Gamma^\Lambda=\Gamma\cap\W, \quad\Gamma^\Lambda_H=\Gamma_H\cap\W;\\
&&\Omega^\Lambda=\Omega\cap \W\times \R,\quad\Omega^\Lambda_H=\Omega_H\cap \W\times\R.
\end{eqnarray*}

\subsection{Introduction to scattering problems}

Given an incident field $u^i$ that satisfies 
\begin{equation}
\Delta u^i+k^2 u^i=0\quad\text{in} \Omega_p,
\end{equation}
then it is scattered by the surface $\Gamma_p$, so there is a scattered field $u^s$ such that it satisfies the Helmholtz equation 
\begin{equation}
\Delta u^s+ k^2 u^s=0\quad \text{ in }\Omega_p
\end{equation}
and the Dirichlet boundary condition
\begin{equation}
u^s=-u^i\quad\text{ on }\Gamma_p.
\end{equation}
\begin{remark}
In this paper, we only take the sound-soft surface as an example. In fact, the results could be extended to problems with some other different settings, for example, the impedance boundary conditions, or inhomogeneous mediums.
\end{remark}
To guarantee that $u^s$ is an out-going wave, i.e., $u^s$ is propagating upwards, it is required that $u^s$ satisfies the {\em Upward Propagating Radiation Condition (UPRC)} in the domain above $\Gamma_H$
\begin{equation}
u^s(x)=\frac{1}{2\pi}\int_\R e^{\i x_1 \xi+\i\sqrt{k^2-|\xi|^2}(x_2-H)}\hat{u}^s(\xi,H)\d\xi\quad x_2\geq H,
\end{equation}
where $\hat{u}^s(\xi,H)$ is the Fourier transform of $u^s(\cdot,H)$. The UPRC is equivalent to 
\begin{equation}
\frac{\partial u^s}{\partial x_2}=T^+ u^s\quad\text{ on }\Gamma_H,
\end{equation}
where $T^+$ is the Dirichlet-to-Neumann map defined by
\begin{equation}
(T^+\phi)(x_1)=\frac{\i}{2\pi}\int_\R \sqrt{k^2-|\xi|^2}e^{\i x_1\xi}\hat{\phi}(\xi)\d\xi,\quad
\end{equation}
where $\hat{\phi}$ is the Fourier transform of the function $\phi$. It has been prove in \cite{Chand2010} that the operator $T^+$ is bounded from $H^{1/2}_r(\Gamma_H)$ to $H^{-1/2}_r(\Gamma_H)$ for any $|r|<1$. Define the total field $u=u^i+u^s$, then it satisfies the following equations
\begin{eqnarray}
\Delta u+k^2 u&=&0\quad\text{ in } \Omega_p\\
u&=&0\quad\text{ on }\Gamma_p\\
\frac{\partial u}{\partial x_2}&=&T^+u+\left[\frac{\partial u^i}{\partial x_2}-T^+u^i\right]\quad\text{ on }\Gamma_H.
\end{eqnarray}
We can formulate the variational problem, i.e., for any $u^i\in H_r^1(\Omega^p_H)$, seek for the weak solution $u\in \widetilde{H}^1_r(\Omega_H^p)$ such that for any test function $\phi\in H^1(\Omega_H^p)$ with a compact support, satisfies the variational equation
\begin{equation}\label{eq:var_origional}
\int_{\Omega_H^p}\left[\nabla u\cdot\nabla \overline{\phi}-k^2u\overline{\phi}\right]\d x-\int_{\Gamma_H}T^+(u|_{\Gamma_H}) \overline{\phi}\d s=\int_{\Gamma_H}\left[\frac{\partial u^i}{\partial x_2}-T^+(u^i|_{\Gamma_H})\right]\overline{\phi}\d s.
\end{equation} 

\begin{remark}
The space $\widetilde{H}^1_r(\Omega_H^p)$ with a tilde is defined as a subspace of $H^1_r(\Omega_H^p)$, i.e., $\widetilde{H}^1_r(\Omega_H^p):=\left\{\phi\in H^1_r(\Omega_H^p):\,\phi\big|_{\Gamma_p}=0\right\}$. In the rest of this paper, we will always use the tilde to indicate the space of functions that satisfies homogeneous Dirichlet boundary conditions. 
\end{remark}

In \cite{Chand2010}, the uniquely solubility of the variational form \eqref{eq:var_origional} was proved.
\begin{theorem}
For any $|r|<1$, given any $u^i\in H^1_r(\Omega_H^p)$, there is a unique $u\in \widetilde{H}^1_r(\Omega_H^p)$ of the variational problem \eqref{eq:var_origional}.
\end{theorem}

\subsection{Bloch transformed fields}

 As the Bloch transform only works on periodic domains, we have to transform the variational problem \eqref{eq:var_origional} defined in a locally perturbed periodic domain $\Omega_H^p$ into the one defined in a periodic domain. Following \cite{Lechl2016,Lechl2017}, define a  diffeomorphism $\Phi_p$ such that maps the periodic domain $\Omega_{H_0}$ to $\Omega_{H_0}^p$, where $\max\left\{\|\zeta\|_\infty,\|\zeta_p\|_\infty\right\}<H_0<H$, by
\begin{equation*} 
\Phi_p(x_1,x_2)=\left(x_1,x_2+\frac{(x_2-H)^3}{(\zeta(x_1)-H)^3}(\zeta_p(x_1)-\zeta_(x_1)\right),\quad (x_1,x_2)^\top\in \Omega_{H_0}.
\end{equation*}
The operator $\Phi_p-I_2$ (where $I_2$ is the identity operator in $\R^2$) has a support which is a subset of $\Omega^\Lambda_{H_0}\subset\Omega^\Lambda_H$. Then let $u_T=u\circ\Phi_p$, then it satisfies the following variational problem
\begin{equation}\label{eq:var_trans}
\begin{aligned}
&&\int_{\Omega_H}\left[A_p\nabla u_T\cdot\nabla\overline{\phi_T}-k^2 c_p u_T\overline{\phi_T}\right]\d x-\int_{\Gamma_H}T^+(u_T|_{\Gamma_H})\overline{\phi_T}\d s\\
&&\quad=\int_{\Gamma_H}\left[\frac{\partial u^i}{\partial x_2}-T^+(u^i|_{\Gamma_H})\right]\overline{\phi_T}\d s
\end{aligned}
\end{equation}
for all $\phi_T\in \widetilde{H}^1(\Omega_H)$ with compact support, where
\begin{eqnarray*}
&&A_p(x)=|\det\nabla\Phi_p(x)|\left[(\nabla\Phi_p(x))^{-1}((\nabla\Phi_p(x))^{-1})^T\right]\in L^\infty(\Omega_H,\R^{2\times 2}),\\
&&c_p(x)=|\det\nabla\Phi_p(x)|\in L^\infty(\Omega_H).
\end{eqnarray*}
From the definition of $\Phi_p$, the supports of $A_p-I$ and $c_p-1$ are both subsets of $\Omega^\Lambda_H$.

Suppose $r\geq 0$, let $w=\J_\Omega u_T$, from the mapping property of the Bloch transform and $u_T\in H_r^1(\Omega_H)\subset H^1(\Omega_H)$, $w\in L^2(\Wast;\widetilde{H}^1_\alpha(\Omega^\Lambda_H))$, and satisfies the following variational problem for any $\phi\in L^2(\Wast;\widetilde{H}^1_\alpha(\Omega^\Lambda_H))$
\begin{equation}\label{eq:var_bloch}
\begin{aligned}
\int_\Wast a_\alpha(w(\alpha,\cdot),\phi(\alpha,\cdot))\d\alpha+b(\J^{-1}_\Omega w,\J^{-1}_\Omega \phi)=\int_\Wast\int_{\Gamma^\Lambda_H}f(\alpha,\cdot)\overline{\phi}(\alpha,\cdot)\d s\d\alpha,
\end{aligned}
\end{equation}
where 
\begin{eqnarray*}
&&\begin{aligned}
a_\alpha(w(\alpha,\cdot),\phi(\alpha,\cdot))&=\int_{\Omega^\Lambda_H}\left[\nabla w(\alpha,\cdot)\cdot\nabla\overline{\phi}(\alpha,\cdot)-k^2w(\alpha,\cdot)\overline{\phi}(\alpha,\cdot)\right]\d x\\&-\int_{\Gamma_H^\Lambda}T^+_\alpha\left[w(\alpha,\cdot)\Big|_{\Gamma_H^\Lambda}\right]\overline{\phi}(\alpha,\cdot)\d s
\end{aligned}\\
&&b(\xi,\psi)=\left[\frac{\Lambda}{2\pi}\right]^{1/2}\int_{\Omega^\Lambda_H}(A_p-I)\nabla\xi\cdot\nabla\overline{\psi}\d s-k^2\left[\frac{\Lambda}{2\pi}\right]^{1/2}\int_{\Omega^\Lambda_H}(c_p-1)\xi\overline{\psi}\d x,\\
&&f(\alpha,\cdot)=\frac{\partial \J_\Omega u^i(\alpha,\cdot)}{\partial x_2}-T^+_\alpha\left[(\J_\Omega u^i)(\alpha,\cdot)|_{\Gamma^\Lambda_H}\right],
\end{eqnarray*}
where the $\alpha$-quasi-periodic DtN map $T^+_\alpha$ is defined by
\begin{equation}
T^+_\alpha(\psi)=\i\sum_{j\in\Z}\sqrt{k^2-|\Lambda^* j-\alpha|^2}\hat{\psi}(j)e^{\i(\Lambda^*j-\alpha)x_1},\quad \psi=\sum_{j\in\Z}\hat{\psi}(j)e^{\i(\Lambda^*j-\alpha)x_1}.
\end{equation}

At the end of this subsection, we will list some of  the useful results in \cite{Lechl2017}. The first one is the equivalence between the original problem \eqref{eq:var_origional} and the one with Bloch transform \eqref{eq:var_bloch}.

\begin{theorem}\label{th:reg1}
Assume $u^i\in H^1_r(\Omega_H^p)$ for some $r\in[0,1)$, then $u_T\in \widetilde{H}^1_r(\Omega_H)$ satisfies \eqref{eq:var_origional} if and only if $w=\J_\Omega u_T\in H_0^r(\Wast;\widetilde{H}^1_\alpha(\Omega^\Lambda_H))$ satisfies \eqref{eq:var_bloch}.
\end{theorem}

Then the variational problem is uniquely solvable for a Liptschiz interface $\Gamma_p$. 
\begin{theorem}\label{th:reg2}
If $\zeta_p$ is Lipschitz continuous, then \eqref{eq:var_bloch} is uniquely solvable in $H^r_0(\Wast;\widetilde{H}^1_\alpha(\Omega^\Lambda_H))$ for all $u^i\in H^1_r(\Omega_H^p)$, $r\in[0,1)$.
\end{theorem}


There is an equivalent form of \eqref{eq:var_bloch} if the $u^i\in H_r^1(\Omega_H^p)$ for $r\in(1/2,1)$.

\begin{theorem}\label{th:continuous}
If $\zeta_p$ is Lipschitz continuous and $u^i\in H^1_r(\Omega_H^p)$ for $r\in(1/2,1)$, then $w\in H_0^r(\Wast;H^1_\alpha(\Omega^\Lambda_H))$ equivalently satisfies for all $\alpha\in\Wast$ and $\phi_\alpha\in H^1_\alpha(\Omega^\Lambda_H)$ that 
\begin{equation}
a_\alpha(w(\alpha,\cdot),\phi_\alpha)+b(\J^{-1}_\Omega w,\phi_\alpha)=\int_{\Gamma^\Lambda_H}f(\alpha,\cdot)\overline{\phi_\alpha}\d s,
\end{equation}
\end{theorem}

\section{Scattering from non-perturbed periodic surfaces}

In Section \ref{sec:review}, we have shown some known results for the solvability  of the scattering problems from locally perturbed periodic surfaces. The regularity of the Bloch transform of the total field $w(\alpha,\cdot)$  depends on the decaying rate of the incident field. In this section, we will turn to a further study of the structure of the Bloch transformed field, for non-perturbed periodic surfaces. Based on the results in this section, the discussion for perturbed surfaces will be carried out in the next section.

The paper \cite{Kirsc1993} provided a very good inspiration for the problem in this section. Firstly, we will define some sets and functions spaces. Define the set of all singularities with a fixed wave number $k$ and a period $\Lambda$ by
\begin{equation*}
\S:=\{\alpha\in\R:\,\exists\, n\in\Z,\,s.t.,\,|\Lambda^*n-\alpha|=k\},
\end{equation*}
which is, definitely, a countable set located periodically on the real line. Let $\I\subset\R$ be an interval (the case that $\I=\R$ is included), $W\subset\R^2$ be a bounded periodic cell of a periodic domain, and define the space of functions defined in the domain $\I\times\Omega^\Lambda_H$ that depends analytically on the first variable
\begin{equation*}
\begin{aligned}
C^\omega\left(\I;S(W)\right):=\Bigg\{f\in\mathcal{D}'(\I\times W):\,\forall \,\alpha_0\in\I,\,\exists\,\delta>0,\,s.t.,\,\forall\alpha\,\in (\alpha_0-\delta,\alpha_0+\delta)\cap\I,\Bigg.\\\left.
\exists\, C>0,\, f_n\in S(W)\,s.t.,\,f(\alpha,x)=\sum_{n=0}^\infty (\alpha-\alpha_0)^n f_n(x),\, \left\|f_n\right\|_{ S(W)}\leq C^n\right\},
\end{aligned}
\end{equation*}
where $S(W)$ is a Sobolev space defined on $W$. In this paper, $S(W)$ is either $H^n(\Omega^\Lambda_H)$ or $\widetilde{H}^n(\Omega^\Lambda_H)$, $n=0,1,2$. 
Define the set of functions defined in the domain $\I\times W$ that depends $C^n$-continuously on the first variable
\begin{equation*}
\begin{aligned}
C^n(\I;S(W)):=\Bigg\{f\in\mathcal{D}'(\I\times W):\,,\,\forall\alpha\in\I,\,j=0,\dots,n,\,\frac{\partial^j f(\alpha,\cdot)}{\partial\alpha^j}\in S(W),\Bigg.\\\left.\left\|\frac{\partial^j f(\alpha,\cdot)}{\partial\alpha^j}\right\|_{S(W)}\text{ is uniformly bounded for }\alpha\in\I\right\},
\end{aligned}
\end{equation*}
thus we can define the space $C^\infty(\I;S(W))$ in the same way
\begin{equation*}
\begin{aligned}
C^\infty(\I;S(W)):=\Bigg\{f\in\mathcal{D}'(\I\times W):\,,\,\forall\alpha\in\I,\,j=0,1,\dots,\infty,\,\frac{\partial^j f(\alpha,\cdot)}{\partial\alpha^j}\in S(W),\Bigg.\\\left.\left\|\frac{\partial^j f(\alpha,\cdot)}{\partial\alpha^j}\right\|_{S(W)}\text{ is uniformly bounded for }\alpha\in\I\right\},
\end{aligned}
\end{equation*}

With the definitions of the new spaces and sets, Theorem a in \cite{Kirsc1993} could be rewritten in the following form with a fixed wave number $k$.  

\begin{theorem}[\cite{Kirsc1993}, Theorem a]\label{th:kirsch}Suppose the incident field $u^i(\alpha,x)$ is the plain wave, i.e., $u^i(\alpha,x):=e^{\i\alpha x_1-\i\sqrt{k^2-\alpha^2}x_2}$, and it is scattered by a smooth enough sound soft surface.
 The total field, denoted by $u(\alpha,x)$, belongs to the space $C^0\left((-k,k);\widetilde{H}^1(\Omega^\Lambda_H)\right)$.  For any interval $\I\subset(-k,k)\setminus\S$, the total field $u$ belongs to $C^\omega\left(\I;\widetilde{H}^1(\Omega^\Lambda_H)\right)$. Moreover, for any $\alpha_0\in\S\cap(-k,k)$, i.e., there is an $n_0\in\Z$ such that $|\Lambda^*n_0-\alpha|=k$, there is a neighborhood $U$ of $\alpha_0$ and quasi-periodic functions $v,w\in C^\omega\left(U\cap(-k,k); \widetilde{H}^1(\Omega^\Lambda_H)\right)$ such that $u=v+\beta_{n_0}(\alpha)w$, where $\beta_{n_0}=\sqrt{k^2-|\Lambda^* n_0-\alpha|^2}$ with non negative real and imaginary parts.
\end{theorem}

\begin{proof}We will just show the main idea of the proof in \cite{Kirsc1993}.  
The proof is based on the variational form of the function $v(\alpha,x):=e^{-\i\alpha x_1}u(\alpha,x)$, i.e., for any $\phi
\in \widetilde{H}_0^1(\Omega^\Lambda_H)$,
\begin{equation*}
\widetilde{a}_\alpha(v(\alpha,\cdot),\phi)=-2\i\sqrt{k^2-\alpha^2}\int_{\Gamma^\Lambda_H}e^{-\i\sqrt{k^2-\alpha^2}x_2}\overline{\phi}\d s,
\end{equation*}
where 
\begin{equation*}
\widetilde{a}_\alpha(\psi,\phi):=\int_{\Omega^\Lambda_H}\left[\grad_\alpha \psi\cdot\overline{\grad_\alpha \phi}-k^2 v(\alpha,\cdot)\overline{\phi}\right]\d x-\int_{\Gamma_H^\Lambda}\widetilde{T}^+_\alpha \left[\psi\big|_{\Gamma^\Lambda_H}\right]\overline{\phi}\d s,
\end{equation*}
$\grad_\alpha=\grad+\i\alpha(1,0)^\top$, and the operator $\widetilde{T}^+_\alpha$ is the DtN map defined on periodic functions in $\Gamma^\Lambda_H$ by
\begin{equation*}
\widetilde{T}^+_\alpha(\psi)=\i\sum_{j\in\Z}\sqrt{k^2-|\Lambda^*j-\alpha|^2}\hat{\psi}(j)e^{\i\Lambda^* j x_1},\quad \psi=\sum_{j\in\Z}\hat{\psi}(j)e^{\i\Lambda^* j x_1}.
\end{equation*} 
Each term in the variational form depends analytically on $\alpha\in(-k,k)$ except for   the term $\int_{\Gamma_H^\Lambda}\widetilde{T}^+_\alpha \left[v(\alpha,\cdot)\Big|_{\Gamma^\Lambda_H}\right]\overline{\phi}\d s$. The operator $\widetilde{T}^+_\alpha$ depends analytically on $\alpha\in(-k,k)\setminus\S$, thus the solution $v(\alpha,\cdot)$ depends analytically in $\alpha\in(-k,k)\setminus\S$. For each $\alpha_0\in\S$, there is a neighbourhood of $U$ such that $\widetilde{T}^+_\alpha$ could be split into one operator depends analytically on $\alpha\in U$ and the one with a $\sqrt{k^2-|\Lambda^*n_0-\alpha|^2}$-singularity. Thus the singularity of $v(\alpha,\cdot)$ with respect to $\alpha\in U$ could be deduced from the Neumann series.
\end{proof}

\begin{remark}
For any $\alpha_0\in\S$, there is either 1) an $n_0\in\Z$ such that $\left|\Lambda^*n_0-\alpha_0\right|=k$ or 2) two $n_1,\,n_2\in\Z$ such that $\left|\Lambda^*n_1-\alpha_0\right|=\left|\Lambda^*n_2-\alpha_0\right|=k$.

 Take Case 1) for example, i.e., $\left|\Lambda^*n_0-\alpha_0\right|=k$, then either $\Lambda^*n_0-k=\alpha_0$ or $\Lambda^*n_0+k=\alpha_0$. If $\Lambda^*n_0-k=\alpha_0$, $\beta_{n_0}=\sqrt{\alpha-\alpha_0}\cdot\sqrt{2k+\alpha_0-\alpha}$ where the second term, which is independent of $n_0$, is analytic for $\alpha$ in a small enough neighborhood of $\alpha_0$. If $k+\Lambda^*n_0=\alpha_0$, 
$\beta_{n_0}=\i\sqrt{2k-\alpha_0+\alpha}\cdot\sqrt{\alpha-\alpha_0}$, where the first term, which is independent in $n_0$, is analytic for $\alpha$ in a small enough neighborhood of $\alpha_0$ as well. Then for $\alpha$ in a small neighborhood $U$ of $\alpha_0$, the total field $u$ in Theorem \ref{th:kirsch} has the form of 
\begin{equation}\label{eq:kirsch}
u=v+\sqrt{\alpha-\alpha_0}\,w
\end{equation}
where $v,\,w$ are both  functions in $C^\omega\left(U\cap(-k,k);\widetilde{H}^1(\Omega^\Lambda_H)\right)$.

 Similar conclusion could be obtained for the second case. Thus we get a simpler form \eqref{eq:kirsch} of the representation of the regularity form near the singularity $\alpha_0\in\S$. 
\end{remark}

The result in Theorem \ref{th:kirsch} can be extended to a wider family of incident fields with the form $\phi(\alpha,x)=e^{\i\alpha x_1-\i\sqrt{k^2-\alpha^2}x_2}$ for $\alpha\in\R$, where both the plane waves and the evanescent waves are included.

\begin{corollary}\label{th:regul_plain_wave}
If the incident field $\phi(\alpha,x)=e^{\i\alpha x_1-\i\sqrt{k^2-\alpha^2}x_2}$, where $\alpha\in\R$, then the total field $u(\alpha,\cdot)$  belongs to the space $C^0(\R;\widetilde{H}^1(\Omega^\Lambda_H))$ and for any interval $\I\subset \R\setminus\S$, $u\in C^\omega(\I;\widetilde{H}^1_\alpha(\Omega^\Lambda_H))$. Moreover, for any $\alpha_0\in\S$, there is a small neighbourhood $U$ of $\alpha_0$ and a pair of functions $v(\alpha,\cdot),\,w(\alpha,\cdot)\in C^\omega\left(U\cap\I; \widetilde{H}^1(\Omega^\Lambda_H)\right)$ such that 
\begin{equation*}
u(\alpha,\cdot)=v(\alpha,\cdot)+\sqrt{\alpha-\alpha_0}\,w(\alpha,\cdot).
\end{equation*}

\end{corollary}

\begin{proof}
It is easy to extend the result of Theorem \ref{th:kirsch} to the case that $\alpha\in (-\infty,-k)\cup(k,+\infty)$. The only case that may make problem is when $\alpha$ in the neighbourhood of $\pm k$, for the plain wave is no longer analytic in $(-k-\delta,-k+\delta)$ or $(k-\delta,k+\delta)$, where $\delta>0$ is a small enough positive number. 

Let $\alpha\in (k-\delta,k+\delta)$, and $\delta$ is small enough such that $\S\cap(k-\delta,k+\delta)=\{k\}$ then the incident wave
\begin{equation*}
\phi(\alpha,x)=\exp(\i\alpha x_1-\i\sqrt{k^2-\alpha^2}x_2).
\end{equation*}

As $e^{-\i\alpha x_1}\phi(\alpha,x)=e^{-\i\sqrt{k^2-\alpha^2}x_2}$ has the form
\begin{equation*}
e^{-\i\alpha x_1}\phi(\alpha,x)=\cosh(-\i\sqrt{k^2-\alpha^2}x_2)+\sinh(-\i\sqrt{k^2-\alpha^2}x_2).
\end{equation*}
Define the functions
\begin{eqnarray*}
&&\phi_1(\alpha, x)=e^{\i\alpha x_1}\cosh(-\i\sqrt{k^2-\alpha^2}x_2),\\
&&\phi_2(\alpha, x)=- e^{\i\alpha x_2}\widetilde{\rm{sinc}}(-\i\sqrt{k^2-\alpha^2}x_2)\sqrt{\alpha+k}\, x_2,
\end{eqnarray*}
where $\widetilde{\rm{sinc}}$ is a function defined in $\C$ by
\begin{equation*}
\widetilde{\rm{sinc}}(z)=\begin{cases}
\frac{\sinh(z)}{z}\quad & z\neq 0\\
1\quad & z=0
\end{cases}.
\end{equation*}
As $\cosh$ and $\widetilde{\rm{sinc}}$ have Taylor's series at $z=0$
\begin{equation*}
\cosh(z)=\sum_{n=0}^\infty \frac{z^{2n}}{(2n)!},\quad
\widetilde{\rm{sinc}}(z)=\sum_{n=0}^\infty\frac{z^{2n}}{(2n+1)!},
\end{equation*}
the functions $\cosh(-\i\sqrt{k^2-\alpha^2}x_2)$  and $\widetilde{\rm{sinc}}(-\i\sqrt{k^2-\alpha^2}x_2)$ has the Taylor's expansion
\begin{eqnarray*}
&&\cosh(-\i\sqrt{k^2-\alpha^2}x_2)=\sum_{n=0}^\infty \frac{(\alpha^2-k^2)^n x_2^{2n}}{(2n)!}\\
&&\widetilde{\rm{sinc}}(-\i\sqrt{k^2-\alpha^2}x_2)=\sum_{n=0}^\infty \frac{(\alpha^2-k^2)^n x_2^{2n}}{(2n+1)!},
\end{eqnarray*}
thus they are both analytic functions in $\alpha$ in a small neighbourhood of $k$. Then the incident field is written into the form of 
\begin{equation*}
\phi(\alpha,x)=\phi_1(\alpha,x)+\sqrt{\alpha-k}\, \phi_2(\alpha,x).
\end{equation*}
From the proof of Theorem \ref{th:kirsch} in \cite{Kirsc1993}, the total field with the incident field $\phi_j\in C^\omega((k-\delta,k+\delta);H^1(\Omega^\Lambda_H))$, $j=1,2$, $u_j$ has the decomposition that
\begin{equation*}
u_j=v_j+\sqrt{\alpha-k}\, w_j
\end{equation*}
where $v_j,\,w_j\in C^\omega\left((k-\delta,k+\delta);\widetilde{H}^1_\alpha(\Omega^\Lambda_H)\right)$ are all analytic functions in $\alpha$. Then the total field $u$ satisfies
\begin{equation*}
\begin{aligned}
u&=u_1+u_2\\
&=v_1+\sqrt{\alpha-k}\, w_1+\sqrt{\alpha-k}\, \left[v_2+\sqrt{\alpha-k}\, w_2\right]\\
&=\left[v_1+(\alpha-k)w_2\right]+\sqrt{\alpha-k}\, \left[w_1+v_2\right].
\end{aligned}
\end{equation*}
The terms $v_1+(\alpha-k)w_2$ and $w_1+v_2$ are both in the space $C^\omega\left((k-\delta,k+\delta);\widetilde{H}^1(\Omega^\Lambda_H)\right)$, let $v=v_1+(\alpha-k)w_2$ and $w=w_1+v_2$, then
\begin{equation*}
u=v+\sqrt{\alpha-k}\, w,
\end{equation*}
the case that $\alpha_0=k$ is proved. The case that $\alpha_0=-k$ could be proved similarly. The proof is finished. 
\end{proof}

In the rest of this section, we will consider the scattering problems with non-perturbed periodic surfaces. Let $u\in\widetilde{H}_r^1(\Omega^H)$ be the total field with the incident $u^i\in H_r^1(\Omega^H)$ and $w=\J_\Omega u\in H_0^r(\Wast;\widetilde{H}^1_\alpha(\Omega^\Lambda_H))$. As $w$ is $\Lambda^*$-periodic in $\alpha$, we only consider the regularity in one periodic cell $\Wast$. For any $\alpha_0\in\S$, $\alpha_0+\Lambda^*\in\S$, thus $\S$ distributes periodically in $\R$. Let $\alpha_1\in\S$ be a fixed number, there are finite number of elements of $\S$ that lies in $[\alpha_1,\alpha_1+\Lambda^*]$. Let $S$ be the number, then $\alpha_1<\alpha_2<\cdots<\alpha_S=\alpha_1+\Lambda^*$ where $[\alpha_1,\alpha_S]\cap\S=\{\alpha_1,\dots,\alpha_S\}$. 

We will define the space of the functions defined in $\I\times W$ that satisfy the property of $u(\alpha,\cdot)$ in  Corollary \ref{th:regul_plain_wave} as follows. 
\begin{equation*}
\begin{aligned}
&\A^\omega(\I;S(W);\S):=\left\{u\in C^0(\I;S(W)):\,\text{ for any  subinterval $\I_0\subset\I\setminus\S$, }\right.\\&u\in C^\omega(\I_0;S(W));\,\text{$\forall\alpha_j\in\I\cap\S$, there is a neighbourhood $U$ of $\alpha_j$}\\&\left.\text{and a pair $v,w\in C^\omega(U\cap\I,S(W))$ such that $u=v+\sqrt{\alpha-\alpha_j}\, w$} \right\}.
\end{aligned}
\end{equation*}

\begin{assumption}\label{asp:incident}
We assume that the incident field $u^i\in H_r^1(\Omega_H)$ satisfies that the Bloch transform  $(\J_\Omega u^i)(\alpha,\cdot)\in \A^\omega\left(\Wast;H^1(\Omega^\Lambda_H);\S\right)$.
\end{assumption}

\begin{remark} Assumption \ref{asp:incident} is satisfied by a large number of incident waves.  In fact, the Green's functions in a half-space satisfies Assumption \ref{asp:incident}. We will take this for example. In the following examples, we set $\Lambda=2\pi$ thus $\Lambda^*=1$.

 The half-space Green's function has the form of
\begin{equation*}
u^i(x,y)=\frac{\i}{4}\left[H_0^{(1)}(k|x-y|)-H_0^{(1)}(k|x-y'|)\right]
\end{equation*}
where $y=(y_1,y_2)^T$ lies above $\Gamma_H$. Bloch transform of the Green's function has the representation of
\begin{equation*}
(\J u^i)(\alpha,x)=\frac{1}{2\pi}\sum_{j\in\Z}e^{\i\alpha_j (x_2-y_2)+\i\beta_j y_2}\rm{sinc}(\beta_j x_2)x_2.
\end{equation*}
As the number of elements in $S_N$ is finite, suppose the number to be $M$, then for each $j=1,\dots,M$, define
\begin{eqnarray*}
&&u_j(\alpha,x)=\frac{1}{2\pi}\sum_{j\in\Z}e^{\i\alpha_j (x_2-y_2)}e^{\i\beta_j y_2}\rm{sinc}(\beta_j x_2)x_2,\\
&&u^1_j(\alpha,x)=\frac{\i}{2\pi}\sum_{j\in\Z}e^{\i\alpha_j (x_2-y_2)}\rm{sinc}(\beta_j y_2)y_2\rm{sinc}(\beta_j x_2)x_2,\\
&&u^2_j(\alpha,x)=\frac{1}{2\pi}\sum_{j\in\Z}e^{\i\alpha_j (x_2-y_2)}\cos(\beta_j y_2)\rm{sinc}(\beta_j x_2)x_2.
\end{eqnarray*}
then $u_j=\beta_j u_j^1+u_j^2$. As $\rm{sinc}$ and $\cos$ has the Taylor's expansions that
\begin{equation*}
\rm{sinc}(x)=\sum_{n=0}^\infty\frac{(-1)^n x^{2n}}{(2n+1)!},\quad
\cos(x)=\sum_{n=0}^\infty \frac{(-1)^n x^{2n}}{(2n)!},
\end{equation*}
$u_j^1$, $u_j^2$ are analytic functions in both $\alpha$ and $x$. 
Note that $(\J_\Omega u^i)(\alpha,x)-\sum_{j=1}^M u_j(\alpha,x)$ is the convergent sum of analytic functions, $\J_\Omega u^i\in \A^\omega\left(\Wast;H^1(\Omega^\Lambda_H);\S\right)$.

\end{remark}

With Assumption \ref{asp:incident}, the regularity result for the Bloch transform of the total field comes directly from Corollary \ref{th:regul_bloch}. 

\begin{theorem}\label{th:regul_bloch}
Let $r>1/2$. Suppose the incident filed $u^i\in H_r^1(\Omega_H)$ satisfies Assumption \ref{asp:incident}, with a total field $u\in \widetilde{H}_r^1(\Omega_H)$. Then the Bloch transform of the total field $w(\alpha,x)=\left(\J_\Omega u\right)(\alpha,x)\in \A^\omega\left(\Wast;\widetilde{H}^1(\Omega^\Lambda_H);\S\right)$. 
\end{theorem}

\begin{proof}
When $r>1/2$, from Theorem \ref{th:continuous}, the variational form is equivalent to the decoupled system 
\begin{equation*}
a_\alpha(w(\alpha,\cdot),\phi_\alpha)=\int_{\Gamma^\Lambda_H}f(\alpha,\cdot)\overline{\phi_\alpha}\d s.
\end{equation*}
Then the periodic function $v(\alpha,x):=e^{\i\alpha x_1}w(\alpha,x)$ satisfies the variational form for any periodic function $\phi\in\widetilde{H}^1_0(\Omega^\Lambda_H)$
\begin{equation*}
\widetilde{a}_\alpha(v(\alpha,\cdot),\phi)=\int_{\Gamma^\Lambda_H}e^{\i\alpha x_1}f(\alpha,\cdot)\overline{\phi}\d s.
\end{equation*}
Thus the conclusion could be obtained from the same argument of the proof of Corollary \ref{th:regul_plain_wave}.
\end{proof}

With the regularity results for non-perturbed periodic surfaces, we will discuss the locally perturbed cases in the next section.

\section{Scattering from locally perturbed periodic surfaces}

In this section, we still hold the assumption that $r>1/2$, then from Theorem \ref{th:continuous}, the variational form is equivalent to 
\begin{equation*}
a_\alpha(w(\alpha,\cdot),\phi_\alpha)+b(u_T,\phi_\alpha)=\int_{\Gamma^\Lambda_H}f(\alpha,\cdot)\overline{\phi_\alpha}\d s.
\end{equation*}
Let $v(\alpha,x)=e^{\i\alpha x_1}w(\alpha,\cdot)$, $\phi=e^{\i\alpha x_1}\phi_\alpha(x)$, then $v$ is the solution of the following system
\begin{equation}\label{eq:var_equi_pert}
\widetilde{a}_\alpha(v(\alpha,\cdot),\phi)=\int_{\Gamma^\Lambda_H}e^{\i\alpha x_1}f(\alpha,\cdot)\overline{\phi}\d s-\widetilde{b}_\alpha(u_T,\phi),
\end{equation}
where the sesquilinear form $\widetilde{b}_\alpha$ is defined by
\begin{equation*}
\widetilde{b}_\alpha(\xi,\psi)=\left[\frac{\Lambda}{2\pi}\right]^{1/2}\int_{\Omega^\Lambda_H}e^{\i\alpha x_1}\Big[(A_p-I)\grad\xi\cdot\overline{\grad_\alpha\psi}-k^2(c_p-1)\xi\overline{\psi}\Big]\d x,
\end{equation*}
which depends analytically on $\alpha$, thus the system \eqref{eq:var_equi_pert} has a right hand side that depends analytically on $\alpha$. So from similar arguments in Theorem \ref{th:regul_bloch}.

\begin{theorem}\label{th:singularity}
Suppose $u^i\in H_r^1(\Omega_H^p)$ for some $r>1/2$ satisfies Assumption \ref{asp:incident}, then the Bloch transform  $w=\J_\Omega u_T\in H_0^r(\Wast;\widetilde{H}^1(\Omega^\Lambda_H))$ belongs to the space $\A^\omega\left(\Wast;\widetilde{H}^1(\Omega^\Lambda_H);\S\right)$.
\end{theorem}

A direct corollary of Theorem \ref{th:singularity} is described as follows.
\begin{corollary}\label{th:singularity_smooth}
Suppose $\alpha_1,\dots,\alpha_s\in\S\cap\overline{\Wast}$, there are $S+1$ functions defined in $\Wast\times\Omega^\Lambda_H$, i.e.,  $w_1(\alpha,\cdot),\dots,w_S(\alpha,\cdot),w_{S+1}(\alpha,\cdot)$ that belong to the space $C^\infty\left(\Wast;\widetilde{H}^1(\Omega^\Lambda_H)\right)$ such that
\begin{equation}
w(\alpha,\cdot)=w_{S+1}(\alpha,\cdot)+\sum_{n=1}^S\sqrt{\alpha-\alpha_n}\,w_n(\alpha,\cdot).
\end{equation}
\end{corollary}


\section{Further regularity results of scattering problems}

In this section, we will consider a more generalized class of incident fields. In the following part, we only want to show the regularity results without proving them, for the proofs are quite similar.

Although the Green's function $u^i(x,y)$ satisfies Assumption \ref{asp:incident}, there is a large class of incident fields that does not satisfy this assumption. We will firstly study the regularity property of the Bloch transform of the Herglotz wave function. Some of the results are from \cite{Lechl2016}. Define the weighted Hilbert space $L^2_{\cos}(-\pi/2,\pi/2)$ as the closure of $C_0^\infty(-\pi/2,\pi/2)$ in the norm
\begin{equation*}
\left\|\phi\right\|_{L^2_{\cos}(-\pi/2,\pi)}=\left[\int_{-\pi/2}^{\pi/2}\left|\phi(\theta)\right|^2/\cos\theta\d\theta\right]^{1/2}.
\end{equation*}
 The downward propagating Herglotz wave function with the density in $L^2_{\cos}(-\pi/2,\pi/2)$ is defined by
\begin{equation*}
(Hg)(x)=\int_{\S_-}e^{\i k x\cdot d}g(d)\d s(d)=\int_{\pi/2}^{\pi/2}e^{\i k\left(\sin\theta x_1-\cos\theta x_2\right)}\phi(\theta)\d\theta,
\end{equation*}
where $\phi\in L^2_{\cos}(-\pi/2,\pi/2)$, $\S_-$ is the lower half of one unit circle in $\R^2$. From \cite{Lechl2015e}, the Bloch transform of the Herglotz wave functions has the representation of 
\begin{equation*}
\left(\J_\Omega Hg\right)(\alpha,x)=\sqrt{\Lambda^*}\sum_{|\Lambda^*j-\alpha|<k}e^{\i(\Lambda^*j-\alpha)x_1-\i\sqrt{k^2-|\Lambda^*j-\alpha|^2}x_2}\frac{\phi(\arcsin(\Lambda^*j-\alpha)/k)}{\sqrt{k^2-|\Lambda^*j-\alpha|^2}}
\end{equation*}
For some choice of the function $\phi$, i.e., if 
$\phi(t)=h\left[\cos(t)\right]\cos(t)$, where $h$ is an analytic function defined in $[0,1]$ with $h(0)=0$, then
\begin{equation*}
\left(\J_\Omega Hg\right)(\alpha,x)=\frac{\sqrt{\Lambda^*}}{k}\sum_{|\Lambda^*j-\alpha|<k}e^{\i(\Lambda^*j-\alpha)x_1-\i\sqrt{k^2-|\Lambda^*j-\alpha|^2}x_2}h\left[\frac{\sqrt{k^2-|\Lambda^*j-\alpha|^2}}{k}\right].
\end{equation*}
By differentiating $\J_\Omega Hg$ with respect to $\alpha$,
\begin{equation*}
\begin{aligned}
&\frac{\partial}{\partial\alpha}\left(\J_\Omega Hg\right)(\alpha,x)\\
=&-\frac{\i\sqrt{\Lambda^*}}{k} x_1\sum_{|\Lambda^*j-\alpha|<k}e^{\i(\Lambda^*j-\alpha)x_1-\i\sqrt{k^2-|\Lambda^*j-\alpha|^2}x_2} h\left[\frac{\sqrt{k^2-|\Lambda^*j-\alpha|^2}}{k}\right]\\
&-\frac{\i\sqrt{\Lambda^*}}{k}x_2\sum_{|\Lambda^*j-\alpha|<k}e^{\i(\Lambda^*j-\alpha)x_1-\i\sqrt{k^2-|\Lambda^*j-\alpha|^2}x_2}\frac{\Lambda^*j-\alpha}{\sqrt{k^2-|\Lambda^*j-\alpha|^2}}h\left[\frac{\sqrt{k^2-|\Lambda^*j-\alpha|^2}}{k}\right]\\
&+\frac{\sqrt{\Lambda^*}}{k^2} \sum_{|\Lambda^*j-\alpha|<k}e^{\i(\Lambda^*j-\alpha)x_1-\i\sqrt{k^2-|\Lambda^*j-\alpha|^2}x_2}\frac{\Lambda^*j-\alpha}{\sqrt{k^2-|\Lambda^*j-\alpha|^2}}h'\left[\frac{\sqrt{k^2-|\Lambda^*j-\alpha|^2}}{k}\right],
\end{aligned}
\end{equation*}
then $\frac{\partial}{\partial\alpha}\left(\J_\Omega Hg\right)(\alpha,x)\in L^p\left(\Wast;H^1(\Omega^\Lambda_H)\right)$ for any $1\leq p<2$, thus the function $\J_\Omega Hg\in W^{1,p}_0\left(\Wast;H^1(\Omega^\Lambda_H)\right)$. From the Sobolev embedding theorem, $\J_\Omega Hg\in H^r_0\left(\Wast;H^1(\Omega^\Lambda_H)\right)$ for some $r\in[1/2,1)$. Thus from the property of the inverse Bloch transform, $Hg\in H_r^1(\Omega_H)$.

As in each open interval $\I\subset\Wast\setminus\S$, the terms in the finite sum do not change, and each term is an analytic function in both $\alpha$ and $x$. Suppose $\alpha_0\in\S$ is a singularity point, then when $\delta>0$ is small enough such that $(\alpha_0-\delta,\alpha_0+\delta)\cap\S=\{\alpha_0\}$. Consider the domain $(\alpha_0-\delta,\alpha_0]$, it is easy to prove that $\left(\J_\Omega Hg\right)(\alpha,x)=v(\alpha,x)+\sqrt{\alpha-\alpha_0}\,w(\alpha,x)$ where $v,w\in C^\omega\left([\alpha_0-\delta,\alpha_0];H^1(\Omega^\Lambda_H)\right)$. Similar result could be obtained in the interval $[\alpha_0,\alpha_0+\delta)$.

We will define the following space.
\begin{equation*}
\begin{aligned}
&\A^\omega_c\left(\I;S(W);\S\right):=\left\{u\in C^0\left(\I;S(W)\right):\text{ for any subinterval }\I_0\subset\I\setminus\S,\right.\\
&u\in C^\omega\left(\I_0;S(W)\right);\,\forall\alpha_j\in\I\cap\S,\text{ there is a small enough }\delta>0\text{ and two pairs }\\
&v_1,w_1\in C^\omega([\alpha_j-\delta,\alpha_j];S(W))\text{ and }v_2,w_2\in C^\omega([\alpha_j,\alpha_j+\delta];S(W))\text{ such that}\\
&\left.u=v_1+\sqrt{\alpha-\alpha_j}\,w_1\text{ for }\alpha\in(\alpha_j-\delta,\alpha_j];\,u=v_2+\sqrt{\alpha-\alpha_j}\,w_2\text{ for }\alpha\in[\alpha_j,\alpha_j+\delta).\right\},
\end{aligned}
\end{equation*}
then the Herglotz wave function discussed above belongs to this space.
\begin{assumption}\label{asp:incident1}
We assume that the incident field $u^i\in H_r^1(\Omega^H_p) $ for $r>1/2$ satisfies that the Bloch transform $\J_\omega u^i\in \A^\omega_c\left(\Wast;H^1(\Omega^\Lambda_H);\S\right)$.
\end{assumption}

 With the incident field satisfying Assumption \ref{asp:incident1}, we can obtain the following corollary of Theorem \ref{th:singularity}.

\begin{corollary}
Suppose $u^i$ satisfies Assumption \ref{asp:incident1}, then the total field $u$ with the Bloch transform $w=\J_\Omega u\in \A^\omega_c\left(\Wast;\widetilde{H}^1(\Omega^\Lambda_H);\S\right)$.  
\end{corollary}

Recall the definition of the set $\S$ of singularities, and let $\alpha_j$, $j\in\Z$ be an ascending series of $\S$, then we can get the corollary of the regularity result in each interval $(\alpha_j,\alpha_{j+1})$.

\begin{corollary}
Suppose $u^i$ satisfies Assumption \ref{asp:incident1}, then the Bloch transform $w=\J_\Omega u$ belongs to the space $C^0\left(\Wast;\widetilde{H}^1(\Omega^\Lambda_H)\right)$. Moreover, for each $j\in\Z$, for $\alpha$ in the interval $[\alpha_j,\alpha_{j+1}]$, there are three functions  $w^j_0,\,w^j_1,\,w^j_2\in C^\infty\left([\alpha_j,\alpha_{j+1}];\widetilde{H}^1(\Omega^\Lambda_H)\right)$ such that
\begin{equation}
w=w^j_0+\sqrt{\alpha-\alpha_j}\,w^j_1+\sqrt{\alpha-\alpha_{j+1}}w^j_2.
\end{equation}
\end{corollary}


\section*{Appendix: Bloch transform}

A useful tool to handle the (locally perturbed) periodic scattering problems is the Bloch transform, see \cite{Lechl2016}, which build up a "bridge" between the infinitely defined problem and a family of  coupled quasi-periodic problems defined in one  single periodic cell. For any function $\phi\in C_0^\infty(\Omega)$, the one dimensional Bloch transform is defined as
\begin{equation}
\left(\J_\Omega\phi\right)(\alpha,x)=\left[\frac{\Lambda}{2\pi}\right]^{1/2}\sum_{j\in\Z}\phi\left(x+\left(\begin{matrix}
\Lambda j\\0
\end{matrix}\right)\right)e^{-\i\alpha\Lambda j},\quad \alpha\in\R, \,x_1\in\R,
\end{equation}
where $\Omega$ is a $\Lambda$-periodic (where $\Lambda>0$) strip  or curve in $x_1$-direction in $\R^2$, which satisfies that if $(x_1,x_2)^T\in\Omega$, then $(x_1\pm\Lambda,x_2)^T\in\Omega$. 
\begin{remark}
Though $\Omega$ does not always have the same definition throughout this paper, we all denote the one dimensional Bloch transforms by $\J_\Omega$. 
\end{remark}
Thus $\J_\Omega\phi(\alpha,x)$ is $\alpha$-quasiperiodic in $x_1$ for period $\Lambda$, for a fixed $\alpha$, i.e.,
\begin{equation*}
\left(\J_\Omega\phi\right)\left(\alpha,x+\left(\begin{matrix}
\Lambda\\0
\end{matrix}\right)\right)=e^{\i\alpha\Lambda}\left(\J_\Omega\phi\right)\left(\alpha,x\right),
\end{equation*}
 and $2\pi/\Lambda$-periodic in $\alpha\in\R$ for a fixed $x$, i.e.
\begin{equation*}
\left(\J_\Omega\phi\right)(\alpha+2\pi/\Lambda,x)=\left(\J_\Omega\phi\right)(\alpha,x).
\end{equation*} 
  Let $\Lambda^*=2\pi/\Lambda$ and set
\begin{equation*}
\W=\left(-\frac{\Lambda}{2},\frac{\Lambda}{2}\right],\quad\Wast=\left(-\frac{\Lambda^*}{2},\frac{\Lambda^*}{2}\right]=\left(-\frac{\pi}{\Lambda},\frac{\pi}{\Lambda}\right],
\end{equation*} 
and let $\Omega^\Lambda=\W\times\R\cap\Omega$. Then $\J_\Omega$ maps a function in $C_0^\infty(\Omega)$ into a function that belongs to the space $C^\infty(\Wast\times\Omega^\Lambda)$. 

To introduce more about the properties of the Bloch transform $\J_\Omega$, we need some notations and spaces. In the rest of this section, $\Omega$ is assumed to be bounded in $x_2$-direction. Firstly, for any $s,r\in\R$, define the weighted space
\begin{equation*}
H_r^s(\Omega)=\left\{\phi\in \mathcal{D}'(\Omega):\,x\rightarrow\left(1+|x_1|^2\right)^{r/2}\phi(x)\in H^s(\Omega)\right\}
\end{equation*}
with the norm
\begin{equation*}
\|\phi\|_{H_r^s(\Omega)}=\left\|\left(1+|x_1|^2\right)^{r/2}\phi(x)\right\|_{H^s(\Omega)}.
\end{equation*}
We can also define the function space in $\Wast\times\Omega^\Lambda$. For some integers $r\in\N$ and some real number $s\in\R$, define the space
\begin{equation*}
H^r_0(\Wast;H^s_\alpha(\Omega^\Lambda))=\left\{\phi\in\mathcal{D}'(\Wast\times\Omega^\Lambda):\,\sum_{\ell=0}^r\int_\Wast\left\|\partial^\ell_\alpha\phi(\alpha,\cdot)\right\|^2_{H^s(\Wast)}\d\alpha\right\}
\end{equation*}
with the norm defined as
\begin{equation*}
\|\phi\|_{H_0^r(\Wast;H^s_\alpha(\Omega^\Lambda))}=\left[\sum_{\ell=0}^r\int_\Wast\left\|\partial^\ell_\alpha\phi(\alpha,\cdot)\right\|^2_{H^s(\Wast)}\d\alpha\right]^{1/2}.
\end{equation*}
With the interpolation in $r$, the definition could be extended to all $r\geq 0$. We can also define the spaces with a negative $r$ by duality arguments, thus the space $H^r_0(\Wast;H^s_\alpha(\Omega^\Lambda))$ is well defined for any $r,\,s\in\R$.   With these definitions, the following properties for Bloch transform holds. For proofs see Theorem 8 in \cite{Lechl2016}, and we will omit them here.

\begin{theorem}
The Bloch transform extends to an isomorphism between $H^s_r(\Omega)$ and $H^r_p(\Wast;H^s_\alpha(\Omega^\Lambda))$ for all $s,r\in\R$. When $s=r=0$, $\J_\Omega$ is an isometry. The inverse operator 
\begin{equation}
\left(\J_\Omega^{-1}w\right)\left(x+\left(\begin{matrix}
\Lambda j\\0
\end{matrix}\right)\right)=\left[\frac{\Lambda}{2\pi}\right]^{1/2}\int_\Wast w(\alpha,x)e^{\i\Lambda j\alpha}\d\alpha,\quad x\in\Omega^\Lambda,\,j\in\Z.
\end{equation}
Moreover, the $L^2$-adjoint operator $\J_\Omega^*$ equals to its inverse $\J_\Omega^{-1}$.
\end{theorem}

\section*{Acknowledgements.} The author was supported by the University of Bremen and the European Union FP7 COFUND under grant agreement n$^\circ{}\,$600411.

\bibliographystyle{alpha}
\bibliography{ip-biblio} 

\providecommand{\noopsort}[1]{}
\begin{thebibliography}{CWRZ99}

\bibitem[Bao94]{Bao1994}
G.~Bao.
\newblock A uniqueness theorem for an inverse problems in periodic diffractive
  optics.
\newblock {\em Inverse Problems}, 10(2):335--340, 1994.

\bibitem[BD00]{Bao2000}
G.~Bao and D.~C. Dobson.
\newblock On the scattering by a biperiodic structure.
\newblock {\em Proc. Amer. Math. Soc.}, 128:2715--2723, 2000.

\bibitem[BDC95]{Bao1995a}
G.~Bao, D.~C. Dobson, and J.~A. Cox.
\newblock Mathematical studies in rigorous grating theory.
\newblock {\em Journal of the Optical Society of America A}, 12(5):1029--1042,
  1995.

\bibitem[CE10]{Chand2010}
S.~N. {Chandler-Wilde} and J.~Elschner.
\newblock Variational approach in weighted {S}obolev spaces to scattering by
  unbounded rough surfaces.
\newblock {\em SIAM. J. Math. Anal.}, 42:2554--2580, 2010.

\bibitem[Coa12]{Coatl2012}
J.~Coatl{\'e}ven.
\newblock {Helmholtz equation in periodic media with a line defect}.
\newblock {\em {J. Comp. Phys.}}, 231:1675--1704, 2012.

\bibitem[CWRZ99]{Chand1999}
S.N. Chandler-Wilde, C.R. Ross, and B.~Zhang.
\newblock Scattering by infinite one-dimensional rough surfaces.
\newblock {\em Proceedings of the Royal Society {A}}, 455:3767--3787, 1999.

\bibitem[CWZ98a]{Chand1998a}
S.~N. Chandler-Wilde and B.~Zhang.
\newblock Electromagnetic scattering by an inhomogeneous conducting or
  dielectric layer on a perfectly conducting plate.
\newblock {\em Proc. R. Soc. Lond. A}, 454:519--542, 1998.

\bibitem[CWZ98b]{Chand1998}
S.~N. Chandler-Wilde and B.~Zhang.
\newblock A uniqueness result for scattering by infinite dimensional rough
  surfaces.
\newblock {\em SIAM J. Appl. Math.}, 58:1774--1790, 1998.

\bibitem[CWZ99]{Chand1999a}
S.~N. Chandler-Wilde and B.~Zhang.
\newblock Scattering of electromagnetic waves by rough interfaces and
  inhomogeneous layers.
\newblock {\em SIAM J. Math. Anal.}, 30:559--583, 1999.

\bibitem[HLP88]{Hardy1988}
G.~H. Hardy, J.~E. Littlewood, and G.~P{\'{o}}lya.
\newblock {\em Inequalities}.
\newblock Cambridge Mathematical Library. Cambridge University Press, 2nd
  edition, 1988.

\bibitem[HN16]{Hadda2016}
H.~Haddar and T.~P. Nguyen.
\newblock {A volume integral method for solving scattering problems from
  locally perturbed infinite periodic layers}.
\newblock {\em Appl. Anal.}, 96(1):130--158, 2016.

\bibitem[Kir93]{Kirsc1993}
A.~Kirsch.
\newblock Diffraction by periodic structures.
\newblock In L.~P{\"a}varinta and E.~Somersalo, editors, {\em Proc. Lapland
  Conf. on Inverse Problems}, pages 87--102. Springer, 1993.

\bibitem[Lec17]{Lechl2016}
A.~Lechleiter.
\newblock The {F}loquet-{B}loch transform and scattering from locally perturbed
  periodic surfaces.
\newblock {\em J. Math. Anal. Appl.}, 446(1):605--627, 2017.

\bibitem[LN15]{Lechl2015e}
A.~Lechleiter and D.-L. Nguyen.
\newblock {Scattering of {H}erglotz waves from periodic structures and mapping
  properties of the {B}loch transform}.
\newblock {\em {Proc. Roy. Soc. Edinburgh Sect. A}}, 231:1283--1311, 2015.

\bibitem[LZ16]{Lechl2016b}
A.~Lechleiter and R.~Zhang.
\newblock Non-periodic acoustic and electromagnetic scattering from periodic
  structures in 3d.
\newblock {\em Submitted}, 2016.

\bibitem[LZ17a]{Lechl2016a}
A.~Lechleiter and R.~Zhang.
\newblock A convergent numerical scheme for scattering of aperiodic waves from
  periodic surfaces based on the {F}loquet-{B}loch transform.
\newblock {\em SIAM J. Numer. Anal}, 55(2):713--736, 2017.

\bibitem[LZ17b]{Lechl2017}
A.~Lechleiter and R.~Zhang.
\newblock A {F}loquet-{B}loch transform based numerical method for scattering
  from locally perturbed periodic surfaces.
\newblock {\em Accepted for SIAM J. Sci. Comput.},
  https://arxiv.org/abs/1611.06360, 2017.

\end{thebibliography}

\end{document}